\documentclass{birkmult}
 \newtheorem{thm}{Theorem}[section]
 \newtheorem{cor}[thm]{Corollary}
 \newtheorem{lem}[thm]{Lemma}
 
 \theoremstyle{definition}
 
 \theoremstyle{remark}
 \newtheorem{rem}[thm]{Remark}
 
 \numberwithin{equation}{section}

\begin{document}
%
%
%
%
%
%
%
%
%
\title[On Ramanujan's lost notebook]
 {On a pair of identities from Ramanujan's lost notebook}
\author[McLaughlin]{James McLaughlin}

\address{%
Department of Mathematics\\
West Chester University\\
West Chester, PA 19383\\
USA}

\email{jmclaughlin@wcupa.edu}

\author[Sills]{Andrew V. Sills}
\address{Department of Mathematical Sciences\br
Georgia Southern University\br
Statesboro, GA 30460-8093\br
USA}
\email{ASills@GeorgiaSouthern.edu}


\date{\today}

\begin{abstract}
Using a pair of two variable series-product identities recorded by
Ramanujan in the lost notebook as inspiration, we find some new
identities of similar type.  Each identity immediately implies an
infinite family of Rogers-Ramanujan type identities, some of which
are well-known identities from the literature.

We also use these identities to derive some general identities for integer partitions.
\end{abstract}

\maketitle

\section{Introduction}
Ramanujan recorded the following identity at the top of a page
of his lost notebook~\cite[p. 33]{R88} (cf.~\cite[p. 99, Entry 5.3.1]{AB09}):
\begin{equation} \label{R1}
  \sum_{n=0}^\infty \frac{x^{2n^2} (-ax;x^2)_n (-x/a;x^2)_n}{(x^2;x^2)_{2n}}
 = \frac{f(ax^3, x^3/a)}{f(-x^2)},
\end{equation}
where we employ the standard notations for rising $q$-factorials,
\[ (A;q)_\infty := (1-A)(1-Aq)(1-Aq^2)\cdots
 \mbox{ and }(A;q)_n:=  \frac{(A;q)_\infty}{(Aq^n;q)_\infty},\]
\[  (A_1, A_2, \dots, A_r; q)_n:= (A_1; q)_n (A_2; q)_n \cdots (A_r;q)_n, \]
and Ramanujan's theta function~\cite[p. 17, Eq. (1.4.8)]{AB09} is given by
\begin{equation} \label{RamThetaDef}
f(a,b) :=  \sum_{j=-\infty}^\infty a^{j(j+1)/2} b^{j(j-1)/2}
= (-a, -b, ab ;ab)_\infty
\end{equation}
with Ramanujan's abbreviation~\cite[p. 17, Eq. (1.4.11)]{AB09}
\begin{equation*} \label{RamFDef}
f(-q) := f(-q,-q^2) = (q;q)_\infty.
\end{equation*}

A bit further down the same page, Ramanujan recorded
\cite[p. 103, Entry 5.3.5]{AB09}
\begin{equation}\label{R2}
 \sum_{n=0}^\infty \frac{x^{n^2} (-ax;x^2)_n (-x/a;x^2)_n }{(x;x^2)_n (x^4;x^4)_n} = \frac{f(ax^2, x^2/a)}{\psi(-x)},
\end{equation}
where
\[ \psi(q):= f(q,q^3) = \frac{(q^2;q^2)_\infty}{(q;q^2)_\infty}\]
is another notation frequently used by Ramanujan~\cite[p. 17, Eq. (1.4.10)]{AB09}.

 From an analytic viewpoint,~\eqref{R1} and~\eqref{R2} are valid for $|x|<1$ and
$a\neq 0$.

 These two identities are noteworthy for several reasons.
Firstly, they are summable two variable Rogers-Ramanujan type identities.
  In contrast, in the standard two
variable generalization of the first Rogers-Ramanujan identity,
\begin{equation*}
\sum_{n=0}^\infty \frac{z^n q^{n^2}}{(q;q)_n}
= \frac{1}{(zq;q)_\infty} \sum_{n=0}^\infty \frac{(-1)^n z^{2n}
q^{n(5n-1)/2} (1-zq^{2n}) (z;q)_n}{ (1-z)(q;q)_n},
\end{equation*}
the right hand side reduces to an infinite product only for special values of $z$, e.g.
$z=1$ gives the first Rogers-Ramanujan identity~\cite[p. 328 (2)]{R94},
\begin{equation*}
  \sum_{n=0}^\infty \frac{q^{n^2}}{(q;q)_n} = \frac{1}{(q;q^5)_\infty (q^4;q^5)_\infty},
\end{equation*} while $z=q$ gives the second Rogers-Ramanujan identity~\cite[p. 330 (2)]{R94},
\begin{equation*}
  \sum_{n=0}^\infty \frac{q^{n(n+1)}}{(q;q)_n} = \frac{1}{(q^2;q^5)_\infty (q^3;q^5)_\infty}.
\end{equation*}

  Secondly, both identities contain an infinite number of
Rogers-Ramanujan type
identities as special cases, a number of which appear in the literature,
as summarized in Tables 1 and 2.

\begin{table}[ht] \caption{Special cases of ~\eqref{R1}}
\begin{tabular}{|c|c| c|}
\hline
 $a$ & $x$ & References \\
 \hline\hline
 $ i $ & $\sqrt{q}$ & Ramanujan~\cite[Entry 4.2.10]{AB09};
 Slater~\cite[p. 156, Eq. (48)]{S52} \\ \hline
  $-1$ & $q$ & Ramanujan~\cite[p. 102, Entry 5.3.3]{AB09} \\ \hline
  $ e^{2\pi i /3}$ & $q$ & Ramanujan~\cite[p. 103, Entry 5.3.4]{AB09} \\ \hline
  $q$ & $q$ & Stanton~\cite[p. 61]{S01}\\ \hline
  $ -q^{1/2} $ & $q^{3/2}$ & Bailey~\cite[p. 422, Eq. (1.6)]{B47}, Slater~\cite[p. 156, Eq. (42)]{S52}\\ \hline
  $ -q $ & $ q^2 $& Slater~\cite[p. 157,  Eq. (53)]{S52} \\ \hline
\end{tabular}
\end{table}

\begin{table}[ht] \caption{Special cases of ~\eqref{R2}}
\begin{tabular}{|c|c| c|}
\hline
 $a$ & $x$ & References \\
 \hline\hline
  $ -1 $  & $q$  & Ramanujan~\cite[p. 104, Entry 5.3.6]{AB09}, Slater~\cite[p. 152, Eq. (4)]{S52} \\ \hline
  $ e^{2\pi i/3}$ & $q$ & Ramanujan~\cite[p. 105, Entry 5.3.8]{AB09} \\ \hline
  $ e^{\pi i /3}$ & $q$ & Ramanujan~\cite[p. 106, Entry 5.3.9]{AB09} \\ \hline
  $q^{1/2}$ & $q^2$ & Gessel-Stanton~\cite[p. 197, Eq. (7.24)]{GS83}\\ \hline
  $ -q $ & $q^{3}$ & Dyson~\cite[p. 9, Eq. (7.5)]{B49}\\ \hline
\end{tabular}
\end{table}

\pagebreak

In \cite{McLSZ09b}  a partner to Ramanujan's \eqref{R1} (identity \eqref{R1partner} below) was found. This motivated us to take another look at \eqref{R1} and \eqref{R2} in the light of this new partner. The results of this reexamination include a new proof of \eqref{R1partner}, a partner to \eqref{R2}, another similar general identity, and two families of false theta series identities.

\begin{align}\label{R1partner}
  \sum_{n=0}^\infty \frac{x^{2n(n+1)} (-a;x^2)_{n+1} (-x^2/a;x^2)_{n} }
  { (x^2;x^2)_{2n+1} } & = \frac{ f(a, x^6/a)}{f(-x^2)},\\
  \label{RR22p}
(1+a)\sum_{n=0}^{\infty}\frac{(-ax,-x/a;x^2)_nx^{n^2+2n}}{(x;x^2)_{n+1}(x^4;x^4)_n}
&=   \frac{f(a, x^4/a)}{\psi(-x)} , \\
\label{R1partnerSS}
  \sum_{n=0}^\infty \frac{ x^{n(n+1)/2} (-x;x)_n (-a;x)_{n+1} (-x/a;x)_{n}  }
  {(x;x)_{2n+1}} &= \frac{f(a, x^2/a)}{ \varphi(-x) },\\
  \label{R1partnerFT}
  \sum_{n=0}^\infty \frac{(-1)^n x^{n(n+1)/2}  (-a;x)_{n+1} (-x/a;x)_{n}  }
  {(x^{n+1};x)_{n+1}} &= \sum_{n=0}^\infty (-1)^n x^{n^2+n} \left( a^{-n} + a^{n+1} \right),\\
  \label{rr22cp}
1+(a-1)\sum_{n=1}^{\infty} \frac{(-a x;x)_{n-1}(-1/a,x;x)_n}{(x;x)_{2n}}&x^{n(n+1)/2}(-1)^n\\
&=\sum_{n=0}^{\infty}x^{n^2}(-1)^n(a^n+x^{2n+1}a^{-n-1}). \notag
\end{align}
where
\[ \varphi(-q) := f(-q,-q) = \frac{(q;q)_\infty}{(-q;q)_\infty} \]
is yet another notation used by Ramanujan~\cite[p. 17, Eq. (1.4.9)]{AB09}.


\begin{rem}
Ramanujan's identity~\eqref{R2} and its partner \eqref{RR22p}
follow from Andrews'  $q$-analog of Bailey's $_2F_1(1/ 2)$ sum~\cite[p. 526, Eq. (1.9)]{A73}:
\begin{equation}\label{aqb}
\sum_{n=0}^{\infty}\frac{(b;q)_n(q/b;q)_n c^n q^{n(n-1)/2}}{(c;q)_n(q^2;q^2)_n}
=\frac{(c q/b; q^2)_\infty (b c;q^2)_{\infty}}{(c;q)_{\infty}}.
\end{equation}

However, each of these two identities can be regarded as the first member in an infinite family of identities, and it does not appear that the more general identities can be similarly extended. For example, the second identity in the sequence whose first member is \eqref{RR22p} is the following  (for consistency with other identities, here we replace $a$ with $z$ and $x$ with $q$):
{\allowdisplaybreaks\begin{multline*}
(1 + q^4z^2)\sum_{n=0}^{\infty}\frac{\displaystyle{(-q^4;q^8)_{n+1}
 \left(-z^2q^8,-1/z^2
;q^8\right)_n} q^{4n^2+8n}}{(q^8;q^8)_{2n+1}}\\
+z(1 + q^{12}z^2)\sum_{n=0}^{\infty}\frac{\displaystyle{(-q^4;q^8)_{n+1}
 \left(-z^2q^{16},-1/q^8z^2
;q^8\right)_n} q^{4n^2+8n}}{(q^8;q^8)_{2n+1}}
\\
=(-z ,-q^4/z,q^4\,;q^4)_{\infty}
\frac{(-q^{4};q^{8})_{\infty}}{(q^{8};q^{8})_{\infty}}
\end{multline*}}
We show that each of \eqref{R1}, \eqref{R2}, \eqref{R1partner}, \eqref{RR22p}, and \eqref{R1partnerSS} may be embedded in an infinite sequence of identities, where each of the stated identities is the first member in the respective sequence of identities. See Section \ref{jtpqtpimp} for more on these identities.

Also, each of \eqref{R1}--\eqref{RR22p} gives rise to a quite general family of partition identities, which does not appear to be true for the more general identity. See Section \ref{parsec} for more details.

\end{rem}

\section{Proofs of the identities}
As is often the case, once the existence of
an identity of Rogers-Ramanujan type is discovered,
it is not hard to prove it using standard techniques.

Recall that $\Big( \alpha_n(z,q), \beta_n(z,q) \Big)$ is called a
\emph{Bailey pair relative to $z$} if
\[ \beta_n (z,q) = \sum_{r=0}^n \frac{\alpha_r (z,q)}{(q;q)_{n-r} (zq;q)_{n+r}}. \]
A well-established method of proof for Rogers-Ramanujan type
identities is insertion of a Bailey pair into
an appropriate limiting case of Bailey's lemma.
Since this method is well documented in the literature, we refer
the reader to, e.g., \cite[Chapter 3]{A86}
or \cite[\S 1.2, p. 3ff]{MS08},
for the details.

\begin{lem}[Andrews-Berndt] \label{ABBP}
 $(\alpha_n, \beta_n)$ form a Bailey pair relative to $q$ where
\[ \alpha_n (q,q) = (a^{-n} + a^{n+1})q^{n(n+1)/2} \] and
\[ \beta_n(q,q) = \frac{(-a;q)_{n+1} (-q/a;q)_n}{(q^2;q)_{2n}}. \]
\end{lem}
\begin{proof}
See~\cite[pp. 98--99]{AB09}.
\end{proof}

\begin{lem} \label{NewBP}
 $(\alpha_n, \beta_n)$ form a Bailey pair relative to $1$ where
\[ \alpha_n (1,q) =
  \left\{  \begin{array}{ll}
  1, &\mbox{if $n=0$}\\
  a^{-n} q^{n(n-1)/2} +a^n q^{n(n+1)/2}, &\mbox{if $n>0$}
  \end{array} \right. \] and
\[ \beta_n(1,q) = \frac{ (-a q;q)_{n} (-1/a;q)_{n}}{ (q;q)_{2n}} . \]
\end{lem}
\begin{proof}
See Lemma 3 in \cite{A92}.
\end{proof}

We now prove \eqref{R1partner}, a partner to Ramanujan's identity at \eqref{R1}. We note that a different, less direct proof of this identity was given in \cite{McLSZ09b}.

\begin{thm} For $a\neq 0$ and $|x|<1$, \[
\sum_{n=0}^\infty \frac{x^{2n(n+1)} (-a;x^2)_{n+1} (-x^2/a;x^2)_{n} }
  { (x^2;x^2)_{2n+1} }  = \frac{ f(a, x^6/a)}{f(-x^2)}.
\]
\end{thm}
\begin{proof}
Insert the Bailey pair $\big( \alpha_n(x^2,x^2), \beta_n(x^2,x^2) \big)$
from Lemma~\ref{ABBP} into Eq. (1.2.8) of~\cite[p. 5]{MS08}.
\end{proof}




\begin{thm}
For $a\not = 0$ and $|x|<1$,
\begin{equation*}
(1+a)\sum_{n=0}^{\infty}\frac{(-ax,-x/a;x^2)_nx^{n^2+2n}}{(x;x^2)_{n+1}(x^4;x^4)_n}
=\frac{(-a,-x^4/a,x^4;x^4)_{\infty}}{\psi(-x)}.
\end{equation*}
\end{thm}
\begin{proof}
Recall that if $(\alpha_n(a,q), \beta_n(a,q))$ is a Bailey pair with
respect to $a$, then the case of the Bailey transform used by Slater
\cite{S51} states that
\begin{multline}\label{Seq1}
\sum_{n=0}^{\infty}(y,z;q)_n \left (\frac{aq}{yz}\right)^n
\beta_n(a,q)\\ =\frac{(aq/y,aq/z;q)_{\infty}}{(aq,aq/yz;q)_{\infty}}
\sum_{n=0}^{\infty}\frac{(y,z;q)_n}{(aq/y,aq/z;q)_n}\left
(\frac{aq}{yz}\right)^n \alpha_n(a,q).
\end{multline}
If we set $y=q\sqrt{a}$ and let $z \to \infty$ in \eqref{Seq1} (see
also (3.14) in \cite{McLSZ09b}), the following identity results:
\begin{multline}\label{Seq1aqf}
\sum_{n=0}^{\infty}(q\sqrt{a};q)_n \left (-\sqrt{a}\right)^n
q^{n(n-1)/2}\beta_n(a,q) \\
=\frac{(q\sqrt{a};q)_{\infty}}{(a q;q)_{\infty}}
\sum_{n=0}^{\infty}(1-\sqrt{a} q^n)\left
(-\sqrt{a}\right)^nq^{n(n-1)/2} \alpha_n(a,q).
\end{multline}
Next, set $a=x$, replace $q$ with $x$ and insert the Bailey pair in Lemma \eqref{ABBP} (with $q$ replaced with $x$). Replace $x$ with $x^2$  to get, after some elementary manipulations, that
\begin{multline}\label{r2eq2b}
\sum_{n=0}^{\infty}\frac{(x;x^2)_{n+1}(-a;x^2)_{n+1}(-x^2/a;x^2)_n (-1)^n x^{n^2}}{(x^2;x^2)_{2n+1}}\\
=\frac{1}{\psi(x)}\sum_{n=0}^{\infty}(1-x^{2n+1})(-1)^n x^{2n^2+n}(a^{-n}+a^{n+1})\\
=\frac{(a x,x^3/a,x^4;x^4)_{\infty}+a(x/a,x^3a,x^4;x^4)_{\infty}}{\psi(x)}
\end{multline}
The last identity follows after expanding the second sum into four sums, re-indexing the two sums containing $(-1)^{n+1}$ by replacing $n$ with $n-1$, and then applying the Jacobi Triple Product twice.

Next, replace $x$ with $-x$, $a$ with $ax$ and note that one of the resulting products on the right side of \eqref{r2eq2b} is now identical to the product side of \eqref{R2}. Subtract \eqref{R2}  from \eqref{r2eq2b} and the result follows.
\end{proof}

We next give a proof of \eqref{R1partnerSS}.
\begin{thm} For $a\neq 0$ and $|x|<1$, \[
\sum_{n=0}^\infty \frac{ x^{n(n+1)/2} (-x;x)_n (-a;x)_{n+1} (-x/a;x)_{n}  }
  {(x;x)_{2n+1}} = \frac{f(a, x^2/a)}{ \varphi(-x) }.
\]
\end{thm}
\begin{proof}
Insert the Bailey pair $\big(\alpha_n(x,x), \beta_n(x,x)\big)$
of Lemma~\ref{ABBP}
into Eq. (S2BL) of~\cite[p. 5]{MS08}.
\end{proof}

\begin{rem}  The preceding result also follows from Andrews'  $q$-analog
 of Gauss's $_2F_1(1/ 2)$ sum~\cite[p. 526, Eq. (1.8)]{A73}:
\begin{equation}\label{aqg}
\sum_{n=0}^{\infty}\frac{(a,b;q)_nq^{n(n+1)/2}}{(q;q)_n(a b q;q^2)_n}
=\frac{(a q,b q;q^2)_{\infty}}{(q,a b q;q^2)_{\infty}}.
\end{equation}
\end{rem}

We next prove the family of false theta series identities stated at \eqref{R1partnerFT}.
This family of identities appears to be new.
\begin{thm} For $a\neq 0$ and $|x|<1$,
 \[
 \sum_{n=0}^\infty \frac{(-1)^n x^{n(n+1)/2}  (-a;x)_{n+1} (-x/a;x)_{n}  }
  {(x^{n+1};x)_{n+1}} = \sum_{n=0}^\infty (-1)^n x^{n^2+n} \left( a^{-n} + a^{n+1} \right)
\]
\end{thm}
\begin{proof}
Insert the Bailey pair $\big(\alpha_n(x,x), \beta_n(x,x)\big)$
of Lemma~\ref{ABBP}
into Eq. (FBL) of~\cite[p. 5]{MS08}.
\end{proof}

Finally, we give a proof of the second family of false theta series identities stated at \eqref{rr22cp}. As with the family in the previous theorem, this family also appears to be new.

\begin{thm}
For $a\not = 0$ and $|x|<1$,
\begin{multline*}
1+(a-1)\sum_{n=1}^{\infty} \frac{(-a x;x)_{n-1}(-1/a,x;x)_nx^{n(n+1)/2}(-1)^n}{(x;x)_{2n}}\\
=\sum_{n=0}^{\infty}x^{n^2}(-1)^n(a^n+x^{2n+1}a^{-n-1}).
\end{multline*}
\end{thm}
\begin{proof}
Set $a=1$ in \eqref{Seq1aqf}, replace $q$ with $x$ and insert the Bailey pair in Lemma \eqref{NewBP} (with $q$ replaced with $x$) to get
\begin{equation*}
\sum_{n=0}^{\infty}\frac{(x,-a x, -1/a;x)_n (-1)^n x^{n(n-1)/2}}{(x;x)_{2n}}
=\sum_{n=1}^{\infty}(1-x^n)(-1)^n(x^{n^2-n}a^{-n}+x^{n^2}a^n).
\end{equation*}
Replace $a$ with $a/x$ and re-index one of the resulting sums to get
\begin{multline}\label{rr2ceq2c}
\sum_{n=0}^{\infty}\frac{(x,-a, -x/a;x)_n (-1)^n x^{n(n-1)/2}}{(x;x)_{2n}}=
1+\sum_{n=1}^{\infty}x^{n^2}(-1)^n(a^{-n}-a^n)\\
-\sum_{n=0}^{\infty}x^{n^2+n}(-1)^n(a^{-n}+a^{n+1}).
\end{multline}
Upon noting that the second sum is the right side of \eqref{R1partnerFT},  re-index the left side of \eqref{R1partnerFT} by separating off the $n=0$ term and replacing $n$ with $n+1$, add the resulting series to \eqref{rr2ceq2c} and simplify, re-index the resulting sum by replacing $n$ with $n-1$ and the result follows after some further simple manipulations.
\end{proof}

\section{Special Cases}
Like Ramanujan's identities~\eqref{R1} and~\eqref{R2}, our
identities generalize a number of identities from the literature.

\begin{table}[ht] \caption{Special cases of ~\eqref{R1partner}}
\begin{tabular}{|c|c| c|}
\hline
 $a$ & $x$ & References \\
 \hline\hline
  $ -1 $  & $\sqrt{q}$  & Ramanujan~\cite[p. 87, Entry 4.2.12]{AB09},
Bailey~\cite[p. 72, Eq. (10)]{B35},  \\
& &  Slater~\cite[p. 154, Eq. (22)]{S52} \\ \hline
$-e^{2\pi i/3}$ & $\sqrt{q}$ & Dyson~\cite[p. 434, Eq. (B3)]{B47},
Slater~\cite[p. 161, Eq. (92)]{S52} \\ \hline
$ -i $ & $\sqrt{q}$ & Ramanujan~\cite[p. 254, Eq. (11.3.5)]{AB05},
\\&&
  Slater~\cite[p. 154, Eq. (28)]{S52}\\ \hline
 $q$ & $q$ & Slater~\cite[p. 154, Eq. (27) and p. 161, Eq. (87)]{S52} \\ \hline
 $-q$ & $q^{3/2}$ & Bailey~\cite[p. 422, Eq. (1.7)]{B47},
  \\&&
  Slater~\cite[p. 156, Eq. (40--corrected)]{S52}
 \\ \hline
  $-q^2$ & $q^{3/2}$ & Bailey~\cite[p. 422, Eq. (1.8)]{B47},
   \\&&
   Slater~\cite[p. 156, Eq. (41--corrected)]{S52}
 \\ \hline
 $q$ & $q^2$ & Slater~\cite[p. 157, Eq. (57)]{S52}\\ \hline
  $-q$ & $q^2$ & Slater~\cite[p. 157, Eq. (55)]{S52}\\ \hline

\end{tabular}
\end{table}


Identity~\eqref{R1partnerSS} with $x=q^2$ and $a=q$ yields Slater~\cite[p. 153, Eq. (11)]{S52}.
Identity~\eqref{R1partnerFT} with $x=q^2$ and $a=q$ yields McLaughlin et al.~\cite[Eq. (2.10)]{MSZ09}.

\section{Some Summation Formulae deriving from the Jacobi Triple Product Identity and the Quintuple Product Identity}\label{jtpqtpimp}

We will make use of the following result, which is an immediate consequence of the Jacobi Triple Product identity.
\begin{lem}
Let $m$ be a positive integer.
For $z\in \mathbb{C}$, $z\not = 0$, and $|q|<1$,
\begin{multline}\label{jtp2}
(-z,-q/z,q;q)_{\infty}\\= \sum_{r=0}^{m-1} \left(
- q^{m^2/2}z^m q^{m(r-1/2)}, -\frac{ q^{m^2/2}}{z^m q^{m(r-1/2)}}, q^{m^2};q^{m^2}
\right)_{\infty}q^{r(r-1)/2}z^{r}.
\end{multline}
\end{lem}
\begin{proof}
Set $a=z$ and $b=q/z$ in \eqref{RamThetaDef} to get, after considering sums in the $m$
arithmetic progressions $mk+r$, $0 \leq r <m$, that
\begin{align*}
(-z,-q/z,q;q)_{\infty}&=\sum_{k=-\infty}^{\infty}z^k q^{k(k-1)/2}
=\sum_{r=0}^{m-1}\sum_{k=-\infty}^{\infty}z^{m k+r} q^{(m k+r)(m k+r-1)/2}\\
&=
\sum_{r=0}^{m-1}q^{r(r-1)/2}z^{r} \sum_{k=-\infty}^{\infty}\left(
z^{m} q^{(m^2-m+2mr)/2} \right)^{k} q^{m^2(k^2-k)/2}.
\end{align*}
The result follows after applying \eqref{RamThetaDef} to each of the inner sums.
\end{proof}

\begin{thm}\label{tgen1}
Let $m$ be a positive integer. For $z, q \in \mathbb{C}$, $z, q \not
= 0$, and $|q|<1$,
\begin{multline}\label{geneq1}
\sum_{r=0}^{m-1}q^{3r(r-1)/2}z^r
\sum_{n=0}^{\infty}\frac{q^{m^2n^2}\left(-q^{m^2/2}z^m q^{3
m(r-1/2)},
\displaystyle{\frac{-q^{m^2/2}}{z^m q^{3 m(r-1/2)}}}\,;q^{m^2} \right)_n}{(q^{m^2};q^{m^2})_{2n}}\\
=\frac{(-z,-q^3/z,q^3\,;q^3)_{\infty}}{(q^{m^2};q^{m^2})_{\infty}}.
\end{multline}
\end{thm}

\begin{proof}
Replace $q$ with $q^3$ in \eqref{jtp2}, and then divide both sides of that
identity by $(q^{m^2};q^{m^2})_{\infty}$. Next, use \eqref{R1} to replace
each of the products in the inner sum with the basic hypergeometric series
given by the left side of \eqref{R1} (replace $x$ with $q^{m^2/2}$ and $a$
with $z q^{3m(r-1/2)}$), and the result follows.
\end{proof}

\begin{cor}
For $z, q \not = 0$ and $|q|<1$, there holds
\begin{multline}\label{geneq1a}
\sum_{n=0}^{\infty}\frac{\displaystyle{ \left(-q^5/z^2,-z^2/q
;q^4\right)_n}
q^{4n^2}}{(q^4;q^4)_{2n}}+z\sum_{n=0}^{\infty}\frac{\displaystyle{
\left(-q^5z^2,-1/qz^2;q^4\right)_n} q^{4n^2}}{(q^4;q^4)_{2n}}
\\
=\frac{\left(-z,-q^3/z,q^3\,;q^3\right)_{\infty}}
{\left(q^4\,;q^4\right)_{\infty}}
\end{multline}
and
\begin{multline}\label{geneq1b}
\sum_{n=0}^{\infty}\frac{\displaystyle{ \left(-q^9/z^3,-z^3
;q^9\right)_n}
q^{9n^2}}{(q^9;q^9)_{2n}}
+z\sum_{n=0}^{\infty}\frac{\displaystyle{ \left(-q^9z^3,-1/z^3
;q^9\right)_n}
q^{9n^2}}{(q^9;q^9)_{2n}}\\
+q^3z^2\sum_{n=0}^{\infty}\frac{\displaystyle{
\left(-q^{18}z^3,-1/z^3q^9 ;q^9\right)_n} q^{9n^2}}{(q^9;q^9)_{2n}}
=\frac{\left(-z,-q^3/z,q^3\,;q^3\right)_{\infty}}
{\left(q^9\,;q^9\right)_{\infty}}.
\end{multline}
\end{cor}

\begin{proof}
These are, respectively, the cases $m=2$ and $m=3$ of Theorem \ref{tgen1}.
\end{proof}

\begin{rem} The case $m=1$ of Theorem \ref{tgen1} is of course
Ramanujan's identity at \eqref{R1}, so Theorem \ref{tgen1} may be
regarded as embedding Ramanujan's identity  in an infinite family of
 identities.
\end{rem}

In a similar manner, \eqref{R1partner} leads to the following
result.

\begin{thm}\label{tgen2}
Let $m$ be a positive integer. For $z, q \in \mathbb{C}$, $z, q \not
= 0$, and $|q|<1$,
\begin{multline}\label{geneq2}
\sum_{r=0}^{m-1}q^{3r(r-1)/2}z^r \times \\
\sum_{n=0}^{\infty}\frac{q^{m^2(n^2+n)}\left(-q^{3m^2/2}z^m q^{3
m(r-1/2)};q^{m^2} \right)_{n+1}
\left(\displaystyle{\frac{-1}{q^{m^2/2}z^m q^{3
m(r-1/2)}}}\,;q^{m^2} \right)_n}
{(q^{m^2};q^{m^2})_{2n+1}}\\
=\frac{(-z,-q^3/z,q^3\,;q^3)_{\infty}}{(q^{m^2};q^{m^2})_{\infty}}.
\end{multline}
\end{thm}

\begin{proof}
The proof is omitted, since it essentially mirrors that of Theorem
\ref{tgen1}.
\end{proof}
\begin{cor}
For $z, q \not = 0$ and $|q|<1$, there holds
\begin{multline}\label{geneq2a}
\sum_{n=0}^{\infty}\frac{\displaystyle{ \left(-q^3 z^2;q^4\right)_{n+1}\left(-q/z^2
;q^4\right)_n}
q^{4n^2+4n}}{(q^4;q^4)_{2n+1}}\\+z\sum_{n=0}^{\infty}\frac{\displaystyle{ \left(-q^9 z^2;q^4\right)_{n+1}\left(-1/q^5z^2
;q^4\right)_n}
q^{4n^2+4n}}{(q^4;q^4)_{2n+1}}
=\frac{\left(-z,-q^3/z,q^3\,;q^3\right)_{\infty}}
{\left(q^4\,;q^4\right)_{\infty}}
\end{multline}
and
\begin{multline}\label{geneq2b}
\sum_{n=0}^{\infty}\frac{\displaystyle{ \left(-q^9 z^3;q^9\right)_{n+1}\left(-1/z^3
;q^9\right)_n}
q^{9n^2+9n}}{(q^9;q^9)_{2n+1}}\\
+z\sum_{n=0}^{\infty}\frac{\displaystyle{ \left(-q^{18} z^3;q^9\right)_{n+1}\left(-1/q^9z^3
;q^9\right)_n}
q^{9n^2+9n}}{(q^9;q^9)_{2n+1}}\\
+q^3z^2\sum_{n=0}^{\infty}\frac{\displaystyle{ \left(-q^{27} z^3;q^9\right)_{n+1}\left(-1/q^{18}z^3
;q^9\right)_n}
q^{9n^2+9n}}{(q^9;q^9)_{2n+1}}
=\frac{\left(-z,-q^3/z,q^3\,;q^3\right)_{\infty}}
{\left(q^9\,;q^9\right)_{\infty}}.
\end{multline}
\end{cor}

\begin{proof}
These are, respectively, the cases $m=2$ and $m=3$ of Theorem \ref{tgen2}.
\end{proof}
\begin{rem} The case $m=1$ gives \eqref{R1partner} above, so this
identity is also the first in an infinite family of identities.
\end{rem}

\begin{thm}\label{tgen3}
Let $m$ be a positive integer.  For $z, q \in \mathbb{C}$, $z, q
\not = 0$, and $|q|<1$,
\begin{multline}\label{geneq3}
\sum_{r=0}^{m-1}q^{2r^2}z^r
\sum_{n=0}^{\infty}\frac{q^{m^2n^2}\left(-q^{m^2}z^m q^{4 m r},
\displaystyle{\frac{-q^{m^2}}{z^m q^{4 m r}}}\,;q^{2m^2} \right)_n}
{(q^{m^2};q^{2m^2})_{n}(q^{4m^2};q^{4m^2})_{n}}\\
=(-z
q^2,-q^2/z,q^4\,;q^4)_{\infty}\frac{(-q^{m^2};q^{2m^2})_{\infty}}{(q^{2m^2};q^{2m^2})_{\infty}}.
\end{multline}
\end{thm}

\begin{proof}
The proof is similar to the proof of Theorem \ref{tgen1}. Replace
$q$ with $q^4$ and then $z$ with $z q^2$ in \eqref{jtp2}, and then
multiply both sides of that identity by
\[(-q^{m^2};q^{2m^2})_{\infty}/(q^{2m^2};q^{2m^2})_{\infty}.\] Then
use \eqref{R2} to replace each of the products in the inner sum with
the corresponding basic hypergeometric series on the left side of
\eqref{R2} (replace $x$ with $q^{m^2}$ and $a$ with $z q^{4 m r}$),
and the result follows.
\end{proof}
\begin{rem} Ramanujan's identity at \eqref{R2} is the case $m=1$ of the
 above theorem, placing this identity also in an infinite family of identities.
\end{rem}

\begin{cor}
For $z, q \not = 0$ and $|q|<1$, there holds
\begin{multline}\label{geneq3a}
\sum_{n=0}^{\infty}\frac{\displaystyle{ \left(-q^4/z^2,-z^2q^4
;q^8\right)_n} q^{4n^2}}{(q^4;q^8)_{n}(q^{16};q^{16})_{n}}
+q^2z\sum_{n=0}^{\infty}\frac{\displaystyle{
\left(-1/z^2q^4,-z^2q^{12} ;q^8\right)_n}
q^{4n^2}}{(q^4;q^8)_{n}(q^{16};q^{16})_{n}}
\\
=(-z q^2,-q^2/z,q^4\,;q^4)_{\infty}
\frac{(-q^{4};q^{8})_{\infty}}{(q^{8};q^{8})_{\infty}}
\end{multline}
and
{\allowdisplaybreaks
\begin{multline}\label{geneq3b}
\sum_{n=0}^{\infty}\frac{\displaystyle{ \left(-q^9/z^3,-z^3q^9
;q^{18}\right)_n} q^{9n^2}}{(q^9;q^{18})_{n}(q^{36};q^{36})_{n}}
+q^2z\sum_{n=0}^{\infty}\frac{\displaystyle{
\left(-1/z^3q^3,-z^3q^{21} ;q^{18}\right)_n}
q^{9n^2}}{(q^9;q^{18})_{n}(q^{36};q^{36})_{n}}\\
+q^8z^2\sum_{n=0}^{\infty}\frac{\displaystyle{
\left(-1/z^3q^{15},-z^3q^{33} ;q^{18}\right)_n}
q^{9n^2}}{(q^9;q^{18})_{n}(q^{36};q^{36})_{n}} =(-z
q^2,-q^2/z,q^4\,;q^4)_{\infty}
\frac{(-q^{9};q^{18})_{\infty}}{(q^{18};q^{18})_{\infty}}.
\end{multline}}
\end{cor}

\begin{proof}
These identities are, respectively, the cases $m=2$ and $m=3$ of
Theorem \ref{tgen3}.
\end{proof}

\begin{thm}\label{tgen4}
Let $m$ be a positive integer.  For $z, q \in \mathbb{C}$, $z, q
\not = 0$, and $|q|<1$,
\begin{multline}\label{geneq4}
\sum_{r=0}^{m-1}q^{r^2-r}z^r\\
\times\sum_{n=0}^{\infty}\frac{q^{m^2(n^2+n)/2}\left(-q^{m^2}z^m q^{m (2r-1)};q^{m^2}\right)_{n+1}\left(-q^{m^2},
\displaystyle{\frac{-1}{z^m q^{m (2r-1)}}}\,;q^{m^2} \right)_n}
{(q^{m^2};q^{m^2})_{2n+1}}\\
=(-z,-q^2/z,q^2\,;q^2)_{\infty}\frac{(-q^{m^2};q^{m^2})_{\infty}}{(q^{m^2};q^{m^2})_{\infty}}.
\end{multline}
\end{thm}
\begin{proof}
Replace $q$ with $q^2$ in  \eqref{jtp2}, multiply both sides of the resulting identity by
$(-q^{m^2};q^{m^2})_{\infty}/(q^{m^2};q^{m^2})_{\infty}$, and then use \eqref{R1partnerSS} to replace each of the resulting infinite products with the corresponding series.
\end{proof}

\begin{cor}
For $z, q \not = 0$ and $|q|<1$, there holds
\begin{multline}\label{geneq4a}
\sum_{n=0}^{\infty}\frac{\displaystyle{ \left(-q^2 z^2;q^4\right)_{n+1}\left(-q^4,-q^2/z^2
;q^4\right)_n}
q^{2n^2+2n}}{(q^4;q^4)_{2n+1}}\\+z\sum_{n=0}^{\infty}\frac{\displaystyle{ \left(-q^6 z^2;q^4\right)_{n+1}\left(-q^4,-1/q^2z^2
;q^4\right)_n}
q^{2n^2+2n}}{(q^4;q^4)_{2n+1}}\\
=\frac{\left(-q^4\,;q^4\right)_{\infty}}
{\left(q^4\,;q^4\right)_{\infty}}\left(-z,-q^2/z,q^2\,;q^2\right)_{\infty}
\end{multline}
and
{\allowdisplaybreaks
\begin{multline}\label{geneq4b}
\sum_{n=0}^{\infty}\frac{\displaystyle{ \left(-q^6 z^2;q^9\right)_{n+1}\left(-q^9,-q^3/z^2
;q^9\right)_n}
q^{9(n^2+n)/2}}{(q^9;q^9)_{2n+1}}\\
+z\sum_{n=0}^{\infty}\frac{\displaystyle{ \left(-q^{12} z^2;q^9\right)_{n+1}\left(-q^9,-1/q^3z^2
;q^9\right)_n}
q^{9(n^2+n)/2}}{(q^9;q^9)_{2n+1}}\\
+q^2z^2\sum_{n=0}^{\infty}\frac{\displaystyle{ \left(-q^{18} z^2;q^9\right)_{n+1}\left(-q^9,-1/q^9z^2
;q^9\right)_n}
q^{9(n^2+n)/2}}{(q^9;q^9)_{2n+1}}\\
=\frac{\left(-q^9\,;q^9\right)_{\infty}}
{\left(q^9\,;q^9\right)_{\infty}}\left(-z,-q^2/z,q^2\,;q^2\right)_{\infty}.
\end{multline}}
\end{cor}

\begin{proof}
These are, respectively, the cases $m=2$ and $m=3$ of Theorem \ref{tgen4}.
\end{proof}

\begin{thm}\label{tgen5}
Let $m$ be a positive integer.  For $z, q \in \mathbb{C}$, $z, q
\not = 0$, and $|q|<1$,
\begin{multline}\label{geneq5}
\sum_{r=0}^{m-1}q^{2r^2-2r}z^r(1+z^mq^{2m^2+(4r-2)m})\\
\times\sum_{n=0}^{\infty}\frac{q^{m^2(n^2+2n)}
\left(-q^{m^2};q^{2m^2}\right)_{n+1}
\left(-q^{3m^2+m (4r-2)}z^m,
\displaystyle{\frac{-1}{q^{m^2+m (4r-2)}z^m }}\,;q^{2m^2} \right)_n}
{(q^{2m^2};q^{2m^2})_{2n+1}}\\
=(-z,-q^4/z,q^4\,;q^4)_{\infty}\frac{(-q^{m^2};q^{2m^2})_{\infty}}{(q^{2m^2};q^{2m^2})_{\infty}}.
\end{multline}
\end{thm}
\begin{proof}
Replace $q$ with $q^4$ in  \eqref{jtp2}, multiply both sides of the resulting identity by
$(-q^{m^2};q^{2m^2})_{\infty}/(q^{2m^2};q^{2m^2})_{\infty}$, and then use \eqref{RR22p} to replace each of the resulting infinite products with the corresponding series.
\end{proof}

\begin{cor}
For $z, q \not = 0$ and $|q|<1$, there holds
\begin{multline}\label{geneq5a}
(1 + q^4z^2)\sum_{n=0}^{\infty}\frac{\displaystyle{(-q^4;q^8)_{n+1}
 \left(-z^2q^8,-1/z^2
;q^8\right)_n} q^{4n^2+8n}}{(q^8;q^8)_{2n+1}}\\
+z(1 + q^{12}z^2)\sum_{n=0}^{\infty}\frac{\displaystyle{(-q^4;q^8)_{n+1}
 \left(-z^2q^{16},-1/q^8z^2
;q^8\right)_n} q^{4n^2+8n}}{(q^8;q^8)_{2n+1}}
\\
=(-z ,-q^4/z,q^4\,;q^4)_{\infty}
\frac{(-q^{4};q^{8})_{\infty}}{(q^{8};q^{8})_{\infty}}
\end{multline}
and
\begin{multline}\label{geneq5b}
(1 + q^{12}z^3)\sum_{n=0}^{\infty}\frac{\displaystyle{(-q^9;q^{18})_{n+1}
 \left(-q^{21}z^3,-1/q^3z^3
;q^{18}\right)_n} q^{9n^2+18n}}{(q^{18};q^{18})_{2n+1}}\\
+z(1 + q^{24}z^3)\sum_{n=0}^{\infty}\frac{\displaystyle{(-q^9;q^{18})_{n+1}
 \left(-q^{33}z^3,-1/q^{15}z^3
;q^{18}\right)_n} q^{9n^2+18n}}{(q^{18};q^{18})_{2n+1}}\\
+q^4z^2(1 + q^{36}z^3)\sum_{n=0}^{\infty}\frac{\displaystyle{(-q^9;q^{18})_{n+1}
 \left(-q^{45}z^3,-1/q^{27}z^3
;q^{18}\right)_n} q^{9n^2+18n}}{(q^{18};q^{18})_{2n+1}}\\ =(-z,-q^4/z,q^4\,;q^4)_{\infty}
\frac{(-q^{9};q^{18})_{\infty}}{(q^{18};q^{18})_{\infty}}.
\end{multline}
\end{cor}

\begin{proof}
These identities are, respectively, the cases $m=2$ and $m=3$ of
Theorem \ref{tgen5}.
\end{proof}

For the next results, we make us of the Quintuple Product Identity
(see \cite{C06} for a survey of the various proofs of this
identity).
\begin{equation}\label{qpi1}
\sum_{n=-\infty}^{\infty}q^{3n^2/2+n/2}(z^{3n}-z^{-3n-1}) =
(q,zq,1/z\,;q)_{\infty}(z^2q,q/z^2\,;q^{2})_{\infty},
\end{equation}
or, alternatively,
\begin{multline}\label{qpi2}
(-q^2z^3,-q/z^3,q^3;q^3)_{\infty}-z^{-1}(-q^2/z^3,-qz^3,q^3;q^3)_{\infty}\\
= (q,z q,1/z\,;q)_{\infty}(z^2q,q/z^2\,;q^{2})_{\infty}.
\end{multline}

\begin{thm}\label{t2gen}
If $z \not = 0$ and $|q|<1$, then
\begin{equation}\label{tqpi1}
\sum_{n=0}^{\infty}
\frac{1-(z+1/z)q^n}{1-(z+1/z)}\frac{(-z^3,-1/z^3\,;q)_n\,q^{n^2}}{(q\,;q)_{2n}}
=(z q,q/z\,;q)_{\infty}(q z^2,q/z^2\,;q^2)_{\infty}.
\end{equation}
\end{thm}

\begin{proof}
By setting $x=q^{1/2}$ in \eqref{R1} and then replacing $a$ with, respectively, $
q^{1/2}z^3$ and $ q^{1/2}/z^3$, we get that
\begin{align*}
\sum_{n=0}^{\infty}\frac{ q^{n^2}(-z^3
q,-1/z^3;q)_n}{(q;q)_{2n}}&=\frac{(-q^2z^3,
-q/z^3,q^3\,;q^3)_{\infty}}{(q;q)_{\infty}}\\
z^{-1}\sum_{n=0}^{\infty}\frac{ q^{n^2}(-q/z^3
,-z^3;q)_n}{(q;q)_{2n}}&=z^{-1}\frac{(-q^2/z^3, -q
z^3,q^3\,;q^3)_{\infty}}{(q;q)_{\infty}}.
\end{align*}
The result follows from \eqref{qpi2}, after subtracting these two
identities and slightly rearranging the resulting identity.
\end{proof}

\begin{thm}\label{t3gen}
If $z \not = 0$ and $|q|<1$, then
\begin{equation}\label{tqpi2}
\sum_{n=0}^{\infty}
\frac{\left(1-\left(q z+\frac{1}{z}\right)q^n\right)
\left(-q^2z^3,\frac{-1}{qz^3}\,;q\right)_n\,q^{n^2+n}}{(q\,;q)_{2n+1}}
=(z q,1/z\,;q)_{\infty}(q^3 z^2,q/z^2\,;q^2)_{\infty}.
\end{equation}
\end{thm}

\begin{proof}
The proof is similar to that of the previous theorem, this time
 setting $x=q^{1/2}$ in \eqref{R1partner}  and then replacing $a$ with, respectively, $
q^{2}z^3$ and $ q z^3$. The details are omitted.
\end{proof}

\section{Partition Identities}\label{parsec}
The analytic identities under consideration in this paper also imply some general partition identities.

\begin{thm}
Let $k\geq 3$ and $r<k/2$ be positive integers. Let $A(n)$ count the
number of partitions of $n$ with
\begin{itemize}
 \item distinct parts $\equiv \pm r \pmod k$,
 \item possibly repeating parts $\equiv 0 \pmod k$,
 \item all odd multiples of $k$ from $k$ to the largest occurring odd
multiple of $k$ occur at least once,
 \item the largest occurring even multiple of $k$ (if any) is at most $k$ more
than the largest occurring odd multiple of $k$,
 \item the largest occurring part $\equiv r \pmod k$ is smaller than
half the largest odd multiple of $k$,
 \item the largest occurring part $\equiv -r \pmod k$ is smaller than $k/2$ plus half
the largest odd multiple of $k$.
  \end{itemize}
Let $B(n)$ count the number of partitions of $n$ with
  \begin{itemize}
 \item distinct parts $\equiv \pm (k+r) \pmod{3k}$,
 \item possibly repeating parts $\equiv \pm k \pmod{3k}$.
  \end{itemize}
Then
\[
A(n) = B(n)
\]
for all integers $n\geq 1$.
\end{thm}

\begin{proof}
Set $x=q^{k/2}$ and $a=q^{r-k/2}$ in \eqref{R1} to get (after some simple
manipulation on the product side) the identity
\begin{equation}\label{R1rk1}
\sum_{n=0}^{\infty}\frac{q^{k
n^2}(-q^{r},-q^{k-r};q^{k})_n}{(q^k;q^k)_{2n}}=\frac{(-q^{k+r},-q^{2k-r};q^{3k})_{\infty}}
{(q^{k},q^{2k};q^{3k})_{\infty}}.
\end{equation}
If we write the left side as $\sum_{n=0}^{\infty}A(n)q^n$ and the
write side as $\sum_{n=0}^{\infty}B(n)q^n$, noting that
\[ k+3k+\dots +(2n-1)k = k n^2,
\]we get the result.
\end{proof}

\begin{thm}
Let $k\geq 3$ and $r<k/2$ be positive integers. Let $C(n)$ count the
number of partitions of $n$ with
\begin{itemize}
 \item distinct parts $\equiv k \pm r \pmod{2k}$,
 \item possibly repeating parts $\equiv 0, \pm k (\mod 4 k)$,
 \item all odd multiples of $k$ from $k$ to the largest occurring odd
multiple of $k$ occur at least once,
 \item the largest occurring  multiple of $4k$ (if any) is at most $2k$ more
than twice the largest occurring odd multiple of $k$,
 \item the largest occurring part $\equiv k+r \pmod{2k}$ is smaller than
$k$ plus the largest occurring odd multiple of $k$,
 \item the largest occurring part $\equiv k-r \pmod{2k}$ is smaller than
the largest occurring odd multiple of $k$.
  \end{itemize}
Let $D(n)$ count the number of partitions of $n$ with
  \begin{itemize}
 \item distinct parts $\equiv 2k \pm r \pmod{4k}$,
 \item possibly repeating parts $\equiv  k \pmod{2k}$.
  \end{itemize}
Then
\[
C(n) = D(n)
\]
for all integers $n\geq 1$.
\end{thm}

\begin{proof}
Similarly, if we set $x=q^{k}$ and $a=q^{r}$ in \eqref{R2}, where
once again $k\geq 3$ and $r<k/2$ are positive integers, we get
 the identity
\begin{equation}\label{R2rk1}
\sum_{n=0}^{\infty}\frac{q^{k
n^2}(-q^{k+r},-q^{k-r};q^{2k})_n}{(q^k;q^{2k})_{n}(q^{4k};q^{4k})_{n}}
=\frac{(-q^{2k+r},-q^{2k-r};q^{4k})_{\infty}}
{(q^{k};q^{2k})_{\infty}}.
\end{equation}
The result now follows.
\end{proof}

\begin{thm}
Let $k\geq 3$ and $r<k/2$ be positive integers. Let $E(n)$ count the
number of partitions of $n$ with
\begin{itemize}
 \item distinct parts $\equiv \pm r \pmod k$,
 \item possibly repeating parts $\equiv 0 \pmod k$,
 \item all even multiples of $k$ from $2k$ to the largest occurring
 even
multiple of $k$ occur at least once,
 \item the largest occurring odd multiple of $k$ (if any) is at most $k$ more
than the largest occurring even multiple of $k$,
 \item the largest occurring part $\equiv r \pmod k$ is smaller
 than $k/2$ plus
half the largest even multiple of $k$,
 \item the largest occurring part $\equiv -r \pmod k$ is smaller than  half
the largest even multiple of $k$.
  \end{itemize}
Let $F(n)$ count the number of partitions of $n$ with
  \begin{itemize}
 \item distinct parts $\equiv \pm r \pmod{3k}$,
 \item possibly repeating parts $\equiv \pm k \pmod{3k}$.
  \end{itemize}
Then
\[
E(n) = F(n)
\]
for all integers $n\geq 1$.
\end{thm}

\begin{proof}
This time set $x=q^{k/2}$ and $a=q^{r}$ in
\eqref{R1partner}, where $k\geq 3$ and $r<k/2$ are positive
integers, to get  the identity
\begin{equation}\label{R3rk1}
\sum_{n=0}^{\infty}\frac{q^{k
(n^2+n)}(-q^{r},q^{k})_{n+1}(-q^{k-r};q^{k})_n}{(q^k;q^k)_{2n+1}}
=\frac{(-q^{r},-q^{3k-r};q^{3k})_{\infty}}
{(q^{k},q^{2k};q^{3k})_{\infty}}.
\end{equation}
The result once again follows, after
noting that
\[ 2k+4k+\dots +2nk = k (n^2+n).
\]
\end{proof}

\begin{thm}
Let $k\geq 3$ and $r<k/2$ be positive integers. Let $G(n)$ count the
number of partitions of $n$ with
\begin{itemize}
 \item distinct parts $\equiv k \pm r \pmod{2k}$,
 \item possibly repeating parts $\equiv 0, \pm k (\mod 4 k)$,
 \item all odd multiples of $k$ from $k$ to the largest occurring odd
multiple of $k$ occur at least once,
 \item the largest occurring  multiple of $4k$ (if any) is smaller than twice the largest occurring odd multiple of $k$,
 \item all parts $\equiv k + r \pmod{2k}$ are smaller than
 the largest occurring odd multiple of $k$,
\item all parts $\equiv k - r \pmod{2k}$ are smaller than $-2k$ plus
 the largest occurring odd multiple of $k$ .
  \end{itemize}
Let $H(n)$ count the number of partitions of $n$ with
  \begin{itemize}
 \item distinct parts $\equiv  \pm r \pmod{4k}$, with the part $r$ not occurring,
 \item possibly repeating parts $\equiv  k \pmod{2k}$, with the part $k$ occurring at least once.
  \end{itemize}
Then
\[
G(n) = H(n)
\]
for all integers $n\geq 1$.
\end{thm}

\begin{proof}
This time cancel the $1+a$ factor on both sides of \eqref{RR22p}, set $x=q^{k}$ and $a=q^{r}$, where
once again $k\geq 3$ and $r<k/2$ are positive integers. Multiply both sides of
 the resulting  identity by $q^k$ to get
\begin{equation}\label{RR22rk2}
\sum_{n=0}^{\infty}\frac{q^{k
(n+1)^2}(-q^{k+r},-q^{k-r};q^{2k})_n}{(q^k;q^{2k})_{n+1}(q^{4k};q^{4k})_{n}}
=\frac{q^k(-q^{4k+r},-q^{4k-r};q^{4k})_{\infty}}
{(q^{k};q^{2k})_{\infty}}.
\end{equation}
The result now follows, upon noting that
\[
k +3k + 5k + \dots + (2n+1)k=k(n+1)^2.
\]
\end{proof}


\end{document}